\documentclass[11pt,a4paper,reqno]{amsart}

\usepackage{amsmath,amsthm}
\usepackage{amssymb}

\usepackage{cases}

\numberwithin{equation}{section}

\theoremstyle{theorem}
\newtheorem{theorem}{Theorem}

\theoremstyle{definition}

\theoremstyle{remark}

\numberwithin{theorem}{section}

\usepackage{url}
\usepackage[colorlinks,linkcolor=blue,citecolor=blue]{hyperref}

\title{Bernoulli and Faulhaber}

\author[J.L. Cereceda]{Jos\'e Luis Cereceda}
\address{%
        Collado Villalba, 28400 -- Madrid, Spain}
\email{jl.cereceda@movistar.es}

\begin{document}

\begin{abstract}
In a recent work, Zielinski used Faulhaber's formula to explain why the odd Bernoulli numbers are equal to zero. Here, we assume that the odd Bernoulli numbers are equal to zero to explain Faulhaber's formula.
\end{abstract}

\maketitle

\section{Introduction}

For integers $n \geq 1$ and $m \geq 0$, denote $S_m = \sum_{i=1}^{n} i^m$. It is well-known that $S_m$ can be expressed in the
so-called Faulhaber form (see, e.g., \cite{beardon,cere1,edwards2,knuth,krishna,shirali})
\begin{align}
S_{2m} & = S_2 \big[ b_{m,0} + b_{m,1} S_1 + b_{m,2} S_{1}^2 + \cdots + b_{m,m-1} S_{1}^{m-1} \big], \label{f1} \\
S_{2m+1} & =  S_{1}^2 \big[ c_{m,0} + c_{m,1} S_1 + c_{m,2} S_{1}^2 + \cdots +  c_{m,m-1} S_{1}^{m-1} \big], \label{f2}
\end{align}
where $b_{m,j}$ and $c_{m,j}$ are non-zero rational coefficients for $j =0,1,\ldots,m-1$ and $m \geq 1$. In particular, $S_3 = S_{1}^2$.
We can write \eqref{f1} and \eqref{f2} more compactly as
\begin{align*}
S_{2m} & = S_2  F_{2m}(S_1),  \\
S_{2m+1} & =  S_{1}^2 F_{2m+1}(S_1),
\end{align*}
where both $F_{2m}(S_1)$ and $F_{2m+1}(S_1)$ are polynomials in $S_1$ of degree $m-1$.

Zielinski derived a version of Faulhaber's formula \eqref{f2} for $S_{2m+1}$ (see \cite[Equation (2.5)]{ziel}). Then, by comparing
the terms in $n$ appearing in \cite[Equation (2.5)]{ziel} and in the traditional Bernoulli polynomial formula
\begin{equation*}
S_{2m+1} = \frac{1}{2m+2} \sum_{j=0}^{2m+1} \binom{2m+2}{j} (-1)^{j} B_{j} n^{2m+2-j},
\end{equation*}
and observing that $B_3 = B_5 =0$, he was able to conclude that $B_{2m+1} =0$ for all $m \geq 1$, where $B_0, B_1, B_2, \ldots \,$ are
the Bernoulli numbers.

In this paper, we show that conversely, assuming $B_{2m+1}=0$ for all $m \geq 1$ leads to the Faulhaber formulas in \eqref{f1} and
\eqref{f2}. To this end, we will use the well-known relationship between the power sums $S_m$ and the Bernoulli polynomials $B_m(x)$, namely
\begin{equation}\label{ship}
S_m = \frac{1}{m+1} \big( B_{m+1}(n+1) - B_{m+1} \big), \quad m,n \geq 1.
\end{equation}
We will also use a theorem established in \cite[Theorem]{goehle}, which for convenience we reproduce as follows.

\begin{theorem}[Goehle and Kobayashi \cite{goehle}]\label{th:1}
Let $f^{(i)}(s)$ denote the $i$th derivative of f evaluated at $s$. If f is a polynomial with even degree $n >1$, then $f$ has a line of symmetry at $s$ if and only if $f^{(i)}(s) =0$ for all odd $i$. Similarly, if $f$ is a polynomial with odd degree $n >1$, then $f$ has a point of symmetry at $(s, f(s))$ if and only if $f^{(i)}(s) =0$ for all even $i \geq 2$.
\end{theorem}

\section{Bernoulli Polynomials}

For $m \geq 0$, the classical Bernoulli polynomial $B_m(x)$ in the real variable $x$ is defined by (see, e.g., \cite{apostol})
\begin{equation*}
B_m(x) = \sum_{j=0}^{m} \binom{m}{j} B_j x^{m-j},
\end{equation*}
where $B_0, B_1, B_2, \ldots \,$ are the Bernoulli numbers, and $B_m(0) = B_m$. We write down the following couple of basic
properties of the Bernoulli polynomials that we will use when proving Theorems \ref{th:2} and \ref{th:3} \cite{apostol}:
\begin{align}
B_m \left( \frac{1}{2} \right) & = \big( 2^{1-m} - 1 \big) B_m,  \label{p1}
\intertext{and}
B_{m}^{(k)}(x) & = k! \binom{m}{k} B_{m-k}(x).  \label{p2}
\end{align}
Note that relation \eqref{p1} holds irrespective of whether $m$ is even or odd. In what follows, by $B_{\text{odd}}^{\prime}$
[$B_{\text{odd}}$], we mean every one of the elements of the finite set $\{ B_3, B_5, \ldots, B_{2m-1} \}$ [$\{ B_3, B_5, \ldots,
B_{2m+1} \}$], where $m$ is any arbitrary fixed integer $\geq 2$ [$\geq 1$]. Next, we establish the following theorem.

\begin{theorem}\label{th:2}
Let $U(x)$ denote the quadratic polynomial $U(x) = \frac{1}{2}x(x-1)$. Then, we have
\begin{align}
& B_{\text{odd}}^{\prime} = 0  \,\,  \Leftrightarrow \,\,  B_{2m}(x) = B_{2m} + \sum_{j=2}^{m} \hat{b}_{j}^{(2m)} U(x)^{j},
\quad  m \geq 2;  \label{th1} \\
& B_{\text{odd}} = 0  \,\,  \Leftrightarrow \,\,  B_{2m+1}(x) = \left( x - \frac{1}{2}\right) \sum_{j=1}^{m} \hat{b}_{j}^{(2m+1)}
U(x)^{j}, \quad m \geq 1,  \label{th2}
\end{align}
where $\hat{b}_{2}^{(2m)}, \ldots, \hat{b}_{m}^{(2m)}, \hat{b}_{1}^{(2m+1)}, \ldots, \hat{b}_{m}^{(2m+1)}$ are nonzero
rational coefficients.
\end{theorem}
\begin{proof}
(i) $B_{\text{odd}}^{\prime} =0 \,\Rightarrow\,$ the right side of \eqref{th1}. If $B_{\text{odd}}^{\prime} =0$, from \eqref{p1}, it follows that $B_{2m-i}(\frac{1}{2}) = 0$ for $i =1,3,\ldots, 2m-1$ (we include $i =2m-1$ because $B_{1}(\frac{1}{2}) =0$). From \eqref{p2}, this in turn implies that $B_{2m}^{(i)}(\frac{1}{2}) =0$ for all odd $i$. (Needless to say, because $\deg B_{2m}(x) =2m$, $B_{2m}^{(s)}(x) =0$ for all $s > 2m$). Therefore, Theorem \ref{th:1} tells us that, for all $m \geq 1$, $B_{2m} (x)$ has a line of symmetry at $\frac{1} {2}$. We can then Taylor-expand $B_{2m}(x)$ about $x = \frac{1}{2}$ to get
\begin{equation*}
B_{2m}(x) = \sum_{j=0}^{m} \hat{v}_j^{(2m)} \left( x-\frac{1}{2} \right)^{2j}.
\end{equation*}
Because $(x - \frac{1}{2})^2 = \frac{1}{4} (1 + 8U(x))$, the last expression can be equivalently written as
\begin{equation}\label{th3}
B_{2m}(x) = \sum_{j=0}^{m} \hat{b}_{j}^{(2m)} U(x)^{j},
\end{equation}
for certain coefficients $\hat{b}_{0}^{(2m)}, \hat{b}_{1}^{(2m)}, \ldots, \hat{b}_{m}^{(2m)}$. Clearly, as $B_{2m}(0) = B_{2m}$, we have that $\hat{b}_{0}^{(2m)} = B_{2m}$. On the other hand, for $m \geq 2$, we must have that $B_{2m}^{\prime}(0) = 2m B_{2m-1}(0) = 2m B_{2m-1}=0$. Differentiating \eqref{th3} and evaluating at $x =0$ yields $B_{2m}^{\prime}(0) = -\frac{1}{2} \hat{b}_{1}^{(2m)}$, from which we deduce that $\hat{b}_{1}^{(2m)} =0$. Moreover, because $B_{2m}^{\prime\prime}(0) \neq 0$, from \eqref{th3}, it follows that $\hat{b}_{2}^{(2m)} \neq 0$. Furthermore, because $B_{2m}^{\prime\prime\prime}(0) =0$ and $\hat{b}_{2}^{(2m)} \neq 0$, from \eqref{th3}, it follows that $\hat{b}_{3}^{(2m)} \neq 0$. Continuing in this fashion, it can be shown that $\hat{b}_{2}^{(2m)}, \ldots, \hat{b}_{m}^{(2m)} \neq 0$. All of these coefficients are rational because the Bernoulli numbers are rational.

(ii) The right side of \eqref{th1} $\, \Rightarrow \, B_{\text{odd}}^{\prime} =0$. It is readily verified that $U(x)$ fulfills the symmetry property $U(x +\frac{1}{2}) = U(\frac{1}{2} - x)$. As a consequence, assuming that $B_{2m}(x) = B_{2m} + \sum_{j=2}^{m} \hat{b}_{j}^{(2m)} U(x)^{j}$, it follows that $B_{2m}(x)$ satisfies the same relation $B_{2m}(x + \frac{1}{2}) = B_{2m}(\frac{1}{2}-x)$, which means that $B_{2m}(x)$ has a line of symmetry at $s = \frac{1}{2}$. Therefore, according to Theorem \ref{th:1}, we must have that $B_{2m}^{(i)}(\frac{1}{2}) =0$ for all odd $i$. From \eqref{p2}, this in turn implies that $B_{2m-i}(\frac{1}{2}) =0$ for all odd $i$. Because $m$ is any arbitrary integer $\geq 2$, from \eqref{p1}, we conclude that $B_{\text{odd}}^{\prime} =0$.

(iii) $B_{\text{odd}} =0 \,\Rightarrow\,$ the right side of \eqref{th2}. If $B_{\text{odd}} =0$, from \eqref{p1}, it follows that $B_{2m+1-i}(\frac{1}{2}) = 0$ for $i =0,2,\ldots, 2m$. From \eqref{p2}, this in turn implies that $B_{2m+1}^{(i)}(\frac{1}{2}) =0$ for all even $i \geq 0$ (we include $i =0$ because $B_{2m+1}(\frac{1}{2})$ is proportional to $B_{2m+1} =0$). Therefore, invoking Theorem \ref{th:1}, we conclude that for all $m \geq 1$, $B_{2m+1}(x)$ has a point of symmetry at $(\frac{1}{2},0)$. We can then Taylor-expand $B_{2m+1}(x)$ about $x = \frac{1}{2}$ to get
\begin{equation*}
B_{2m+1}(x) = \sum_{j=0}^{m} \hat{v}_j^{(2m+1)} \left( x-\frac{1}{2} \right)^{2j+1} =  \left( x-\frac{1}{2} \right)
\sum_{j=0}^{m} \hat{v}_j^{(2m+1)} \left( x-\frac{1}{2} \right)^{2j}.
\end{equation*}
As before, because $(x - \frac{1}{2})^2 = \frac{1}{4} (1 + 8U(x))$, the last expression can be equivalently written as
\begin{equation}\label{th4}
B_{2m+1}(x) =  \left( x-\frac{1}{2} \right) \sum_{j=0}^{m} \hat{b}_{j}^{(2m+1)} U(x)^{j},
\end{equation}
for certain coefficients $\hat{b}_{0}^{(2m+1)}, \hat{b}_{1}^{(2m+1)}, \ldots, \hat{b}_{m}^{(2m+1)}$. Clearly, as $B_{2m+1}(0) =B_{2m+ 1}$, we have that, for $m \geq 1$, $\hat{b}_{0}^{(2m+1)} =0$. On the other hand, because $B_{2m+1}^{\prime}(0) = (2m+1)B_{2m}(0) = (2m +1)B_{2m} \neq 0$, from \eqref{th4}, it follows that $\hat{b}_{1}^{(2m+1)} \neq 0$. Similarly, using \eqref{p2}, it can be shown that the rational coefficients $\hat{b}_{1}^{(2m+1)}, \ldots, \hat{b}_{m}^{(2m+1)} \neq 0$.

(iv) The right side of \eqref{th2} $\, \Rightarrow \, B_{\text{odd}} =0$. Assume that $B_{2m+1}(x) = \left( x - \frac{1}{2}\right) \sum_{j=1}^{m} \hat{b}_{j}^{(2m+1)}U(x)^{j}$. Then, because $U(x +\frac{1}{2}) = U(\frac{1}{2} - x)$, it follows that $B_{2m+1}(x +\frac{1}{2}) = -B_{2m+1}(\frac{1}{2} -x)$. This means that $B_{2m+1}(x)$ has a point of symmetry at $(\frac{1}{2}, 0)$. According to Theorem \ref{th:1}, this implies that $B_{2m+1}^{(i)}(\frac{1}{2}) =0$ for all even $i \geq 0$ (we include the case $i =0$ because, from the right side of \eqref{th2}, we have that $B_{2m+1}(\frac{1}{2}) =0$). Therefore, taking into account \eqref{p2} and \eqref{p1}, and noting that $m$ is any arbitrary integer $\geq 1$, we conclude that $B_{\text{odd}} =0$.
\end{proof}

For completeness, we write down an explicit representation for the coefficients $\hat{b}_{j}^{(2m)}$ and $\hat{b}_{j}^{(2m+1)}$ \cite{cere2}:
\begin{align}
\hat{b}_{j}^{(2m)} & = 8^j \sum_{k=j}^{m} \frac{1}{4^k}\binom{2m}{2k}\binom{k}{j} B_{2m-2k}\left( \frac{1}{2} \right),
\label{cof1} \\
\hat{b}_{j}^{(2m+1)} & =  8^j \sum_{k=j}^{m} \frac{1}{4^k}\binom{2m+1}{2k+1}\binom{k}{j} B_{2m-2k}\left( \frac{1}{2}
\right),  \label{cof2}
\end{align}
where $j =0,1,\ldots,m$. It is to be noted that, as we have shown, $b_{0}^{(2m)} = B_{2m}$ (for $m \geq 0$), $b_{0}^{(2m+1)} =0$
(for $m \geq 1$), and $b_{1}^{(2m)} =0$ (for $m \geq 2$).

\section{Bernoulli Meets Faulhaber}

Next, we establish the following theorem, which highlights the close relationship between the property $B_{\text{odd}}=0$ and the Faulhaber formulas in \eqref{f1} and \eqref{f2}.

\begin{theorem}\label{th:3}
For $m \geq 1$, we have that
\begin{equation*}
B_{\text{odd}} =0 \,\, \Leftrightarrow \, \, \begin{cases}
S_{2m} = S_2 F_{2m}(S_1),  \\
S_{2m+1} = S_{1}^2 F_{2m+1}(S_1),
\end{cases}
\end{equation*}
where $F_{2m}(S_1)$ and $F_{2m+1}(S_1)$ are polynomials in $S_1$ of degree $m-1$.
\end{theorem}
\begin{proof}
(i) $B_{\text{odd}} =0 \, \Rightarrow \, S_{2m} = S_2 F_{2m}(S_1)$. According to Theorem \ref{th:2}, if $B_{\text{odd}} =0$, then
$B_{2m+1}(x)$ can be expressed as $B_{2m+1}(x) = (x - \frac{1}{2}) \sum_{j=1}^{m} \hat{b}_{j}^{(2m+1)} U(x)^{j}$. Therefore,
from \eqref{ship}, it follows that
\begin{equation*}
S_{2m} = \frac{1}{2m+1} \left(n + \frac{1}{2} \right) \sum_{j=1}^{m} \hat{b}_{j}^{(2m+1)} S_{1}^{j},
\end{equation*}
because we are assuming that $B_{2m+1} =0$. It is immediate to see that the last equation can be written as
\begin{equation}\label{s2m}
S_{2m} = \frac{3}{4m+2} S_2 \big[ \hat{b}_{1}^{(2m+1)}+ \hat{b}_{2}^{(2m+1)} S_1 + \cdots + \hat{b}_{m}^{(2m+1)}
S_{1}^{m-1} \big],
\end{equation}
which is obviously of the form \eqref{f1}.

(ii) $S_{2m} = S_2 F_{2m}(S_1) \, \Rightarrow \, B_{\text{odd}} =0$. If $S_{2m} = S_2 F_{2m}(S_1)$, from \eqref{ship}, we obtain
\begin{equation}\label{th5}
(2m+1) S_2 F_{2m}(S_1) = B_{2m+1}(n+1) - B_{2m+1}.
\end{equation}
Considering both $S_1$ and $S_2$ as polynomials in the real variable $x$, we have that $S_2 = \frac{1}{3}(2x+1) S_1$, and then $S_2(-\frac{1}{2}) =0$. In view of \eqref{th5}, this implies that $B_{2m+1}(\frac{1}{2}) = B_{2m+1}$. On the other hand, from \eqref{p1}, we have that $B_{2m+1}(\frac{1}{2}) = (2^{-2m} - 1) B_{2m+1}$, from which we deduce that $B_{2m+1} =0$. Because $m$ is any arbitrary integer $\geq 1$, we conclude that $B_{\text{odd}} =0$.

(iii) $B_{\text{odd}} =0 \, \Rightarrow \, S_{2m+1} = S_{1}^2 F_{2m+1}(S_1)$. If $B_{\text{odd}} =0$, then, from Theorem \ref{th:2}, it follows that $B_{2m+2}(x)$ can be expressed as $B_{2m+2}(x) = B_{2m+2} + \sum_{j=2}^{m+1} \hat{b}_{j}^{(2m+2)} U(x)^{j}$ (note that, because we are using $B_{2m+2}(x)$, we have to assume $B_{\text{odd}} =0$ instead of $B_{\text{odd}}^{\prime}=0$ for Theorem \ref{th:2} to apply to this situation). Hence, from \eqref{ship}, we get
\begin{equation}\label{s2m1}
S_{2m+1} = \frac{1}{2m+2} \sum_{j=2}^{m+1} \hat{b}_{j}^{(2m+2)} S_{1}^{j} = \frac{S_{1}^{2}}{2m+2}
\big[ \hat{b}_{2}^{(2m+2)}+ \hat{b}_{3}^{(2m+2)} S_1 + \cdots + \hat{b}_{m+1}^{(2m+2)} S_{1}^{m-1} \big],
\end{equation}
which is obviously of the form \eqref{f2}.

(iv) $S_{2m+1} = S_{1}^2 F_{2m+1}(S_1) \, \Rightarrow \, B_{\text{odd}} =0$. A proof of this statement was given in \cite{ziel}. An alternative proof is as follows: if $S_{2m+1} = S_{1}^2 F_{2m+1}(S_1)$, from \eqref{ship}, we obtain
\begin{equation*}
B_{2m+2}(n+1) = B_{2m+2} + (2m+2) S_1^2 F_{2m+1}(S_1).
\end{equation*}
Considering $S_1 = \frac{1}{2}x(x+1)$ as a polynomial in the real variable $x$, we then have
\begin{equation*}
B_{2m+2}(x) = B_{2m+2} + (2m+2) (U(x))^2 F_{2m+1}(U(x)).
\end{equation*}
Because $U(x +\frac{1}{2}) = U(\frac{1}{2} - x)$, it turns out that $B_{2m+2}(x)$ equally fulfills $B_{2m+2}(x+\frac{1}{2}) = B_{2m+2}(\frac{1}{2} -x)$, and thus, $B_{2m+2}(x)$ has a line of symmetry at $s = \frac{1}{2}$. According to Theorem \ref{th:1}, this implies that $B_{2m+2}^{(i)}(\frac{1}{2}) =0$ for all odd $i$. From \eqref{p2}, this means that $B_{2m+2-i}(\frac{1}{2}) =0$ for all odd $i$. Because $m$ is any arbitrary integer $\geq 1$, from \eqref{p1}, we conclude that $B_{\text{odd}} =0$.
\end{proof}

From \eqref{f1} and \eqref{s2m}, it readily follows that, for $j=0,1,\ldots,m-1$, $b_{m,j} = \frac{3}{4m+2}\hat{b}_{j+1}^{(2m+1)}$. Therefore, from \eqref{cof2}, we obtain
\begin{equation}\label{cofb}
b_{m,j} = \frac{3 \cdot 8^{j+1}}{4m+2} \sum_{k =j+1}^{m} \frac{1}{4^k} \binom{2m+1}{2k+1}\binom{k}{j+1}
B_{2m-2k}\left( \frac{1}{2} \right).
\end{equation}
Similarly, from \eqref{f2} and \eqref{s2m1}, we have that, for $j=0,1,\ldots,m-1$, $c_{m,j} = \frac{1}{2m+2}\hat{b}_{j+2}^{(2m+2)}$. Hence, using \eqref{cof1}, and after a simple rearrangement, we find that
\begin{equation}\label{cofc}
c_{m,j} = \frac{8^{j+1}}{j+2} \sum_{k =j+1}^{m} \frac{1}{4^k} \binom{2m+1}{2k+1}\binom{k}{j+1}
B_{2m-2k}\left( \frac{1}{2} \right).
\end{equation}
Moreover, in view of \eqref{cofb} and \eqref{cofc}, there is a relation between the coefficients $b_{m,j}$ and $c_{m,j}$, namely,
\begin{equation*}
c_{m,j} = \frac{4m+2}{3j+6} b_{m,j}, \quad  j =0,1,\ldots, m-1.
\end{equation*}
Thus, knowing the coefficients $b_{m,j}$ in \eqref{f1} allows us to know the coefficients $c_{m,j}$ in \eqref{f2}, and vice versa. For example, because $S_{2m}^{\prime}(0) = B_{2m}$, from \eqref{f1}, it is seen that $b_{m,0} = 6B_{2m}$, and then the last equation tells us that $c_{m,0} = (4m+2)B_{2m}$.

\section{Conclusion}

In \cite{ziel}, Zielinski wonders why the odd Bernoulli numbers are equal to zero, and answers by saying that it is because
$S_{2m+1}$ is a polynomial in $S_{1}^{2}$, and $S_{1}^{2} = \frac{1}{4}(n^4 + 2n^3 +n^2)$. In this paper, we have shown that
$S_{2m}$ and $S_{2m+1}$ admit the polynomial representation in \eqref{f1} and \eqref{f2}, respectively, just because the odd Bernoulli numbers are equal to zero.


\vspace{2mm}

\begin{thebibliography}{10}


\bibitem{apostol} T. M. Apostol, \emph{A primer on Bernoulli numbers and polynomials}, Math. Mag., \textbf{81.3} (2008), 178--190.


\bibitem{beardon} A. F. Beardon, \emph{Sums of powers of integers}, Amer. Math. Monthly, \textbf{103.3} (1996), 201--213.


\bibitem{cere1} J. L. Cereceda, \emph{Explicit form of the Faulhaber polynomials}, College Math. J., \textbf{46.5} (2015), 359--363.


\bibitem{cere2} J. L. Cereceda, \emph{Power sums of arithmetic progressions and Bernoulli polynomials}, Int. J. Contemp.
Math. Sci., \textbf{14.4} (2019), 187--200.


\bibitem{edwards2} A. W. F. Edwards, \emph{A quick route to sums of powers}, Amer. Math. Monthly, \textbf{93.6} (1986), 451--455.


\bibitem{goehle} G. Goehle and M. Kobayashi, \emph{Polynomial graphs and symmetry}, College Math. J., \textbf{44.1} (2013), 37--42.


\bibitem{knuth} D. E. Knuth, \emph{Johann Faulhaber and sums of powers}, Math. Comp., \textbf{61} (1993) 277--294.


\bibitem{krishna} H. K. Krishnapriyan, \emph{Eulerian polynomials and Faulhaber's result on sums of powers of integers},
College Math. J., \textbf{26.2} (1995), 118--123.


\bibitem{shirali} S. A. Shirali, \emph{On sums of powers of integers}, Resonance, \textbf{12.7} (2007), 27--43.


\bibitem{ziel} R. Zielinski, \emph{Faulhaber and Bernoulli}, The Fibonacci Quarterly, \textbf{57.1} (2019), 32--34.


\end{thebibliography}
\end{document}